\documentclass[a4paper, 12pt, reqno]{amsart}

\usepackage{fancyhdr}
\usepackage{amssymb,amsmath}
\usepackage[dvips]{graphics}
\usepackage[dvips]{color}
\usepackage{graphicx}
\usepackage[english]{babel}
\usepackage[utf8]{inputenc}
\usepackage{comment}
\usepackage{todonotes}
\usepackage{esint}
\usepackage{stmaryrd}
\usepackage{accents}

\usepackage[colorlinks,pdfpagelabels,pdfstartview = FitH,bookmarksopen
= true,bookmarksnumbered = true,linkcolor = blue,plainpages =
false,hypertexnames = false,citecolor = red,pagebackref=false]{hyperref}

\allowdisplaybreaks
\renewcommand{\epsilon}{\varepsilon}

\newtheorem{theorem}{Theorem}[section]
\newtheorem{lemma}[theorem]{Lemma}

\newtheorem{corollary}[theorem]{Corollary}

\theoremstyle{definition}
\newtheorem{definition}[theorem]{Definition}

\newtheorem{remark}[theorem]{Remark}
\newtheorem{construction}[theorem]{Construction}

\numberwithin{equation}{section}

\title[Lower semicontinuous obstacles]{Lower semicontinuous obstacles for the porous medium equation}

 \author[R.\ Korte]{Riikka Korte}
 \address{Riikka Korte\\
 Aalto University, Department of Mathematics and Systems Analysis\\
 P.O. Box 11100, FI-00076 Aalto, Finland}
 \email{riikka.korte@aallto.fi}

 \author[P. Lehtel\"a]{Pekka Lehtel\"a}
 \address{Pekka Lehtel\"a\\
 Aalto University, Department of Mathematics and Systems Analysis\\
 P.O. Box 11100, FI-00076 Aalto, Finland}
 \email{pekka.lehtela@aalto.fi}

 \author[S. Sturm]{Stefan Sturm}
 \address{Stefan Sturm\\
 Fachbereich Mathematik, Universit\"at Salzburg\\
 Hellbrunner Str. 34, 5020 Salzburg, Austria}
 \email{stefan.sturm@sbg.ac.at}

\thanks{The research is partially supported by the Emil Aaltonen Foundation and Academy of Finland, grant number 308063.}

\newcommand\Rn{\mathbb R^n}
\newcommand\R{\mathbb R}
\newcommand\Q{\mathbb Q}
\newcommand\N{\mathbb N}

\newcommand\bd{\partial}
\newcommand\ph{\varphi}
\newcommand\eps{\varepsilon}
\newcommand\F{\mathcal{F}}

\DeclareMathOperator*{\essinf}{ess\,inf}
\newcommand\dx {\, d}

\providecommand{\ch}[1]{\text{\raise 2pt \hbox{$\chi$}\kern-0.2pt}_{#1}}

\providecommand{\vint}[1]{\mathchoice
          {\mathop{\vrule width 5pt height 3 pt depth -2.5pt
                  \kern -9pt \kern 1pt\intop}\nolimits_{\kern -5pt{#1}}}%
          {\mathop{\vrule width 5pt height 3 pt depth -2.6pt
                  \kern -6pt \intop}\nolimits_{\kern -3pt{#1}}}%
          {\mathop{\vrule width 5pt height 3 pt depth -2.6pt
                  \kern -6pt \intop}\nolimits_{\kern -3pt{#1}}}%
          {\mathop{\vrule width 5pt height 3 pt depth -2.6pt
                  \kern -6pt \intop}\nolimits_{\kern -3pt{#1}}}}

\begin{document}

\begin{abstract}
We deal with the obstacle problem for the porous medium equation in the slow diffusion regime $m>1$. Our main interest is to treat fairly irregular obstacles assuming only boundedness and lower semicontinuity. In particular, the considered obstacles are not regular enough to work with the classical notion of variational solutions, and a different approach is needed. We prove the existence of a solution in the sense of the minimal supersolution lying above the obstacle. As a consequence, we can show that non-negative weak supersolutions to the porous medium equation can be approximated by a sequence of supersolutions which are bounded away from zero.
\end{abstract}



\date{\today}  
\maketitle

\section{Introduction}

In this paper, we study the obstacle problem for the porous medium equation 
\begin{equation} \label{eq:PME}
  \partial_t u - \Delta u^m = 0 \quad \text{in } \Omega_T,
\end{equation}
where $\Omega_T=\Omega\times (0,T)$ denotes the space-time cylinder of height $T>0$ over a bounded open set $\Omega\subset\Rn$. We concentrate on the degenerate regime $m>1$, also known as the slow diffusion case. Our main goal here is to prove the existence of a solution to the obstacle problem for fairly irregular obstacles which are only bounded and lower semicontinuous.   As some general references to the theory of the porous medium equation, we mention \cite{daskalopoulos-kenig,vazquez,wu}. 

We define the solution to the obstacle problem as the least supersolution above the given obstacle, see Definition \ref{def_obstacle_solution} for the details. The supersolutions we have in mind are defined in terms of a parabolic comparison principle, analogously to supercaloric functions in classical potential theory. For that reason, we refer to them as $m$-supercaloric functions.

Obstacle problems particularly provide a method of approximation in the following sense. Since (super)solutions are lower semicontinuous by \cite[Thm.\ 1.1]{lower-semicontinuity}, there exists a nondecreasing approximating sequence of smooth functions which can be used as obstacles to get an approximating sequence of supersolutions. 

Generally, in the case of the porous medium equation, the set where the solution vanishes often causes technical difficulties. Therefore, one would like to have a tool of approximation by supersolutions which are bounded away from zero. This issue has already been recognized by DiBenedetto, Gianazza, and Vespri in their discussion of the proof of \cite[Thm.\ V.17.1]{dibenedetto-harnack}. We address this question constructing such an approximating sequence by using perturbed supersolutions as obstacles.  

Previously, the obstacle problem for the porous medium equation has usually been studied from the point of view of variational inequalities, see \cite{alt-luckhaus,regularity-sol, obstacle,holder-obstacle}. The existence of such variational solutions was established in \cite[Thm.\ 2.7]{obstacle}. However, this approach requires sufficiently regular obstacles. In particular, the existence of the time derivative of the obstacle is needed to guarantee the existence of variational solutions. Thus, for instance a supersolution is not an admissible obstacle in this context. 


A different approach to obstacle problems based on minimal supersolutions above the obstacle has been employed in the case of evolutionary $p$-Laplace type equations for instance in \cite{korte-kuusi-siljander}. In \cite[Cor.\ 3.16]{irregular-obstacles}, it was established that the notion of variational solution coincides with the least supersolution as long as the involved obstacle possesses adequate regularity. By contrast, the connection between those two concepts is not well understood for the porous medium equation. By \cite[Thm.\ 4.12]{teemu-benny-comparison}, we know that the least supersolution is a variational solution, provided that the obstacle is regular enough, whereas the other implication remains an open question. 

To the authors' knowledge, the existence question in the sense of least supersolutions has not been addressed in the literature yet. As our main result, Theorem \ref{main_theorem}, we prove the existence of a solution to the obstacle problem for bounded lower semicontinuous obstacles by approximation with continuous obstacles. For the latter, we construct a solution to the obstacle problem by iteratively building a sequence of functions, analogously to the Schwarz alternating method (see \cite{kkp,equivalence-def}, for instance). However, following the approach in \cite{korte-kuusi-siljander}, we need to slightly modify the technique to ensure that the limiting function stays above the obstacle. Note that the function we construct is often called ``Balayage'' in potential theory. 

Finally, as a consequence of our main theorem, we will show in Theorem \ref{not-main_theorem} that a bounded $m$-supercaloric function can be approximated by a nonincreasing sequence of $m$-supercaloric functions which are bounded away from zero. Our result can be applied to improve earlier Harnack type estimates which were often established under additional assumptions like $u>0$ or $u\geq\varepsilon>0$. An approximation tool that allows to bypass such assumptions was previously suggested in \cite{dibenedetto-harnack}, however a rigorous proof was missing up to now.

\section{Preliminaries}

Before we list the definitions and basic tools used in this paper, we clarify the notation. By $U_{t_1,t_2} \Subset \Omega_T$, we describe the fact that the space-time cylinder $U_{t_1,t_2} = U \times (t_1,t_2)$ is compactly contained in $\Omega_T$, i.\,e.\ $\overline{U_{t_1,t_2}} \subset \Omega_T$. The parabolic boundary of $U_{t_1,t_2}$ is 
\[ \bd_p U_{t_1,t_2}= (\partial U \times (t_1,t_2)) \cup (\overline U \times \{t_1\}),
\]
 and we will write $|U_{t_1,t_2}|$ for the $(n+1)$\:\!--\:\!dimensional Lebesgue measure of such a set. On the contrary, space-time boxes 
 \[
 \prod_{i=1}^n (a_i,b_i) \times (t_1,t_2)
 \]
 will usually be called $Q$. Here, $\prod_{i=1}^n (a_i,b_i)$ indicates the set $(a_1,b_1)\times ... \times (a_n,b_n)$. Moreover, we will refer to the superlevel set 
 \[
 \{ (x,t)\in\Omega_T\colon u(x,t)>\psi(x,t) \}
 \] 
 where the function $u$ lies above the obstacle $\psi$ by $\{ u>\psi \}$. Finally, for a function $u$, we will use the abbreviation $u_+=\max\{u,0\}$.

Next, we define our notion of weak (super)solutions.

\begin{definition} \label{def:supersolutions}
 A non-negative function $u\colon \Omega_T \to [0,\infty]$ is a {\it weak supersolution to the porous medium equation} \eqref{eq:PME} if $u^m \in L^2_{\textup{loc}}(0,T; H^1_{\textup{loc}}(\Omega))$ and $u$ satisfies
\begin{equation}\label{weak-super}
\iint_{\Omega_T} \Big(-u\partial_t\ph + \nabla u^m \cdot \nabla \ph \Big) \, dx \, dt \ge 0
\end{equation}
for any test function $\ph \in C_0^\infty (\Omega_T)$ with $\ph \ge 0$. Similarly, $u$ is a {\it weak subsolution} if the above inequality holds reversed. Moreover, $u$ is a {\it weak solution} if it is a weak sub- and supersolution.
\end{definition}

As a tool, we need the boundary value problem 
\begin{equation} \label{eq:BVP}
\begin{cases}
  \partial_t u - \Delta u^m = 0 ~~\text{in } \Omega_T,\\
  u(\cdot,0)=g(\cdot,0) ~~\text{in } \Omega,\\
u^m - g^m \in L^2(0,T; H_0^1(\Omega))
\end{cases}
\end{equation}
with boundary and initial values determined by a sufficiently regular function $g$. We continue by giving the rigorous definition of weak solutions to the boundary value problem.

\begin{definition}
  A non-negative function $u\in C^0([0,T];L^{m+1}(\Omega))$ is a {\it weak solution to the boundary value problem} \eqref{eq:BVP} with sufficiently regular boundary and initial values $g$ if $u^m - g^m \in L^2(0,T; H_0^1(\Omega))$ and $u$ satisfies
\[
\iint_{\Omega_T} \Big(-u\partial_t\ph + \nabla u^m \cdot \nabla \ph \Big) \, dx \, dt = \int_\Omega g(x,0) \ph(x,0) \, dx
\]
for all smooth test functions $\ph$ having compact support with respect to space and vanishing at $t=T$.
\end{definition}

The next defintion concerns the notion of $m$-supercaloric functions. 

\begin{definition} \label{def:supercaloric}
  A function $u\colon\Omega_T\rightarrow [0,\infty]$ is {\it $m$-supercaloric} if 
  \begin{enumerate}
  \item $u$ is lower semicontinuous;
  \item $u$ is finite in a dense subset of $\Omega_T$;
  \item $u$ satisfies the following comparison principle in every interior cylinder $U_{t_1,t_2}\Subset \Omega_T$: 
If $w\in C^0(\overline{U_{t_1,t_2}})$ is a weak solution to the porous medium equation \eqref{eq:PME} in $U_{t_1,t_2}$ and $u\ge w$ on $\bd_p U_{t_1,t_2}$, then $u\ge w$ in $U_{t_1,t_2}$.
  \end{enumerate}
\end{definition}

\begin{remark}\label{supercaloric-is-weak}
  The weak (super)solutions according to Definition \ref{def:supersolutions} are $m$-supercaloric. Indeed, the lower semicontinuity follows from \cite[Thm.\ 1.1]{lower-semicontinuity}, and the comparison principle is due to \cite[Thm.\ 6.5]{vazquez}. Moreover, in the locally bounded case, $m$-supercaloric functions are weak supersolutions (see \cite[Thm.\ 1.3]{supersol}). 
\end{remark}

Finally, we state the precise definition of a solution to the obstacle problem.

\begin{definition} \label{def_obstacle_solution}
Let $\psi\colon \Omega_T\to [0,\infty)$ be a bounded and lower semicontinuous obstacle. A function $u\colon\Omega_T\to [0,\infty]$ is a {\it solution to the obstacle problem for the porous medium equation} with obstacle $\psi$ if
 	\begin{enumerate}
		\item $u\geq\psi$ in $\Omega_T$;
		\item $u$ is an $m$-supercaloric function in $\Omega_T$;
		\item $u$ is the smallest $m$-supercaloric function in $\Omega_T$ which lies above $\psi$, i.\,e.\ if $v$ is an $m$-supercaloric function in $\Omega_T$ with $v\geq\psi$ in $\Omega_T$, then $v\geq u$ in $\Omega_T$.
	\end{enumerate}
\end{definition}

\begin{remark} \label{rmk:local_boundedness}
We mention the following properties of solutions to the obstacle problem, which can be deduced from the previous definition.
\begin{enumerate}
\item[(i)] Note that solutions to the obstacle problem in the above sense are unique by their definition.
\item[(ii)] We will prove in Lemma \ref{property_sol_obst} that solutions $u$ to the obstacle problem are weak solutions to the porous medium equation in the set $\{ u>\psi \}$.
\item[(iii)] Due to the boundedness assumption on the obstacle function $\psi$, we also know that solutions $u$ to the obstacle problem in the sense of the previous definition are locally bounded by $\sup_{\Omega_T} |\psi|$ (see \cite[Thm.\ 3.1]{teemu-benny-comparison} and Lemma \ref{property_sol_obst}). Therefore, since locally bounded $m$-supercaloric functions are weak supersolutions by \cite[Thm.\ 1.3]{supersol}, we conclude that solutions to the obstacle problem are in fact weak supersolutions to the porous medium equation in $\Omega_T$, too.

\end{enumerate}
\end{remark}

To overcome the fact that constants cannot be added to solutions to the porous medium equation, we need the following result, which can be retrieved from \cite[Lemma 3.2]{perron}.
\begin{lemma}\label{cool-convergence}
  Let $g\in C^0(\overline{\Omega_T})$ be such that $g^m \in L^2(0,T;H^1(\Omega))$ and $0\le g \le M$. Suppose that $\eps>0$ and let $u$ and $u_\eps$ be the solutions to the boundary value problem \eqref{eq:BVP} with boundary and initial values given by $g$ and $g_\eps=(g^m+\eps^m)^{1/m}$, respectively. Then, we have
\begin{equation} \label{cool-convergence-ineq}
\iint_{\Omega_T} (u_\eps-u)(u_\eps^m - u^m) \,dx \,dt \le \eps^m |\Omega_T| (M+1) + \eps |\Omega_T| (M+1)^m. 
\end{equation}
\end{lemma}

For the interested reader, we mention that the proof utilizes an Ole\u\i nik type test function 
\[
\ph (x,t)=
\begin{cases}
  \int_t^T (u_\eps^m - u^m - \eps^m) \, ds &\text{for } 0<t<T,\\
0 &\text{for } t\ge T.
\end{cases}
\]
We refer to \cite{perron} for the actual proof. Observe that the argument does not work for general porous medium type equations with coefficients.

\begin{remark}\label{cool-convergence-rmk}
  The previous lemma implies that, for $u_\eps$ and $u$ as above, there exists a (nonrelabeled) subsequence of $u_\eps$ such that $u_\eps \to u$ a.\,e.\ in $\Omega_T$. This can be deduced by the elementary inequality $|u_\eps-u|^{m+1} \le  (u_\eps-u)(u_\eps^m - u^m)$, which can be found in \cite[Cor.\ 3.11]{obstacle}. 
\end{remark}

Next, we consider the Poisson modification $u_P\colon\Omega_T\to [0,\infty]$ of a locally bounded $m$-supercaloric function $u$. It is defined with respect to some (sufficiently regular) space-time cylinder $U_{t_1,t_2}\Subset \Omega_T$ as 
\begin{equation}\label{poisson-def}
u_P =
\begin{cases}
  u &\text{in } \Omega_T\setminus \overline{U_{t_1,t_2}},\\
v &\text{in } \overline{U_{t_1,t_2}},
\end{cases}
\end{equation}
where $v$ is the continuous weak solution in $U_{t_1,t_2}$ with boundary values given by $u$ on $\bd_p U_{t_1,t_2}$. The existence of such a weak solution $v$ will be justified in the following lemma. 

\begin{lemma}\label{poisson-lemma}
  Let $u$ be a locally bounded $m$-supercaloric function in $\Omega_T$. The Poisson modification $u_P$ exists and it is an $m$-supercaloric function. 
\end{lemma}

\begin{proof}
First, we show the existence of the Poisson modification by constructing the function $v$ appearing in \eqref{poisson-def}. As $u$ is lower semicontinuous, there exists a nondecreasing sequence of smooth functions $\eta_k\ge 0$ converging to $u$ pointwise in $\Omega_T$. Let $v_k$ be the continuous weak solution to the boundary value problem \eqref{eq:BVP} in $U_{t_1,t_2}$ with boundary and initial values given by $\eta_k$ on $\bd_p U_{t_1,t_2}$. By the comparison principle, we obtain $v_k \le v_{k+1}\le u$ in $U_{t_1,t_2}$ for every $k\in \N$. Thus, due to the local boundedness of $u$ and the Harnack type convergence theorem \cite[Lemma 3.4]{supersol}, $v=\lim v_k$ is a continuous weak solution in $U_{t_1,t_2}$.  

It remains to prove that $u_P$ is $m$-supercaloric. From the above construction of $v$, it follows that $u_P$ satisfies the finiteness condition (2) in Definition \ref{def:supercaloric}. Next, we will argue that $u_P$ is lower semicontinuous. Obviously, it suffices to consider the points $(x_0,t_0)\in \bd U_{t_1,t_2}$. For those, we calculate
\[
\liminf_{{\substack{(x,t) \to (x_0,t_0)\\(x,t)\in U_{t_1,t_2}}}} u_P(x,t) = \liminf_{{\substack{(x,t) \to (x_0,t_0)\\(x,t)\in U_{t_1,t_2}}}} v(x,t) = v(x_0,t_0)= u_P(x_0,t_0).
\]
Hence, we can conclude the lower semicontinuity of $u_P$. Finally, we will show that $u_P$ satisfies the comparison principle. For that purpose, we let $V_{\tau_1,\tau_2}\Subset \Omega_T$ and consider a weak solution $w\in C^0(\overline{V_{\tau_1,\tau_2}})$. Suppose that $$w\le u_P ~\textup{ on } \bd_pV_{\tau_1,\tau_2}.$$ By construction, we know that  $u_P \le u$, which implies $$w\le u ~\textup{ on } \bd_p V_{\tau_1,\tau_2}.$$ Therefore, by the comparison principle for $u$, we find that $w\le u$ in $V_{\tau_1,\tau_2}$. This takes care of the points $(x,t)\in V_{\tau_1,\tau_2}\setminus U_{t_1,t_2}$, and it remains to consider the set $V_{\tau_1,\tau_2}\cap U_{t_1,t_2}$. The previous consideration particularly shows that $$w\le u=v ~\textup{ on } \bd_pU_{t_1,t_2} \cap V_{\tau_1,\tau_2}.$$ On the other hand, we have $$w\le u_P=v ~\textup{ on } \bd_p V_{\tau_1,\tau_2} \cap U_{t_1,t_2}$$ by assumption. Thus, there holds $$w\le v ~\textup{ on } \bd_p (U_{t_1,t_2}\cap V_{\tau_1,\tau_2}),$$ and the comparison principle implies $$w\le v ~\textup{ in } U_{t_1,t_2}\cap V_{\tau_1,\tau_2}.$$ Together, we can conclude $w\le u_P$ in $V_{\tau_1,\tau_2}$. 
\end{proof}

To complete this section, we establish the following property of solutions to the obstacle problem.

\begin{lemma} \label{property_sol_obst}
Let $\psi\colon \Omega_T\to [0,\infty)$ be a bounded and lower semicontinuous obstacle. Suppose that $u$ is a solution to the obstacle problem in $\Omega_T$. Then, $u$ is a weak solution to the porous medium equation in the set $\{ u>\psi \}$.
\end{lemma}

\begin{proof}
Let $(x_0,t_0)\in \{u> \psi\}$. Since the set $\{u> \psi\}$ is open by the lower semicontinuity of $u$, we can find a number $\lambda >0$ and a neighbourhood $U_{t_1,t_2}\Subset \{u>\psi\}$ of $(x_0,t_0)$ such that $$u > \lambda > \psi ~\textup{ in }\overline{U_{t_1,t_2}}.$$ Recall that $u$ is locally bounded by Remark \ref{rmk:local_boundedness}. Therefore, as the set $U$ can be chosen arbitrarily smooth, we may consider the Poisson modification $u_P$ of $u$ with respect to $U_{t_1,t_2}$, defined as in \eqref{poisson-def}. Since $u_P=u$ on $\bd_p U_{t_1,t_2}$, the comparison principle implies $$u \ge u_P > \lambda > \psi ~\textup{ in } U_{t_1,t_2}.$$ By Lemma \ref{poisson-lemma}, the function $u_P$ is $m$-supercaloric, hence, we know $u=u_P$ in $U_{t_1,t_2}$ from property (3) in Definition \ref{def_obstacle_solution}. Therefore, $u$ is a weak solution to the porous medium equation in $U_{t_1,t_2}$. Finally, since being a solution is a local property, we conclude that $u$ is a weak solution in $\{u>\psi\}$.
\end{proof}

\section{Existence of solutions for continuous obstacles} \label{sec:obstacle_cont}
In this section, we describe how to construct a solution to the obstacle problem with a bounded continuous obstacle. Therefore, we assume throughout this section that $\psi\in C^0(\Omega_T)\cap L^\infty(\Omega_T)$.

\begin{construction} \label{construction}
Our aim is to construct the unique solution to the obstacle problem as the limit of a sequence $(f_j)_{j\in\N_0}$ using a modified Schwarz alternating method. The functions $f_j$ are obtained recursively by solving boundary value problems on a dense countable collection $\F=\{Q_j \subset \Omega_T \colon j \in \N_0\}$ of space-time boxes $Q_j$ ending at $t=T$. For instance, one could consider the collection 
\[
\F = \bigg\{ \prod_{i=1}^n (a_i,b_i) \times (t, T) \subset \Omega_T \colon a_{i}<b_{i}, \, a_i, b_i,t \in \Q\bigg\}.
\] 
Next, we describe the construction in detail. We choose $f_0 = \psi$ in $\Omega_T$, and, for any $j\in \N_0$, we define $f_{j+1}$ as 
\[
f_{j+1} = 
\begin{cases}
 \max\{g_{j}, f_j\}  &\text{in } Q_j,\\
f_j  &\text{in } \Omega_T\setminus Q_j,
\end{cases}
\]
where $g_j$ is the continuous weak solution to the boundary value problem \eqref{eq:BVP} in $Q_j$ with boundary and initial values given by $f_j$. Roughly, the idea there is to redefine $f_j$ as $g_j$ in the box $Q_j$ like in the Schwarz alternating method. However, in order to guarantee that the functions $f_j$ stay above the obstacle, we modify the approach by taking the maximum. Finally, we denote the pointwise limit by
\begin{align} \label{limit_u}
u(x,t) = \lim_{j\to\infty} f_j(x,t).
\end{align}
The limit exists and is well-defined for any $(x,t)\in\Omega_T$, which will be shown in Lemma \ref{lemma32} and Lemma \ref{box-lemma}.
\end{construction}

Next, we collect the basic properties of the previous construction in the following two lemmata.

\begin{lemma} \label{properties_fj}
Let $\psi\in C^0(\Omega_T)\cap L^\infty(\Omega_T)$. Then, the following statements for the generating sequence $(f_j)_{j\in\N_0}$ from Construction \ref{construction} hold:
	\begin{enumerate}
		\item[(i)] The sequence $(f_j)_{j\in\N_0}$ is nondecreasing. More precisely, we have $$f_{j+1} \geq f_j \geq \psi ~\textit{ in } \Omega_T$$ for any $j\in\N_0$.
		\item[(ii)] The sequence $(f_j)_{j\in\N_0}$ is uniformly bounded. More precisely, we have $$|f_j| \leq \sup_{\Omega_T} |\psi| ~\textit{ in } \Omega_T$$ for any $j\in\N_0$. 
		\item[(iii)] The function $f_j$ is continuous for any $j\in\N_0$.
		\item[(iv)] The function $f_j$ is a weak subsolution to the porous medium equation in the set $\{ f_j>\psi \}$ for any $j\in\N$.
	\end{enumerate}
\end{lemma}

\begin{proof}
By definition, we have $f_0= \psi$ and $$f_{j+1}= \max\{g_j,f_j\} \geq f_j$$ for any $j\in\N_0$. This proves assertion (i). Next, we will show (ii). Obviously, there holds $$|f_0|\leq \sup_{\Omega_T}|\psi| ~\textup{ in } \Omega_T.$$ Suppose now that the claim is true for some $j\in\N_0$. Since the functions $f_{j+1}$ are constructed inductively by solving boundary value problems with boundary and initial values given by $f_j$, the comparison principle gives (ii). Since the obstacle $\psi$ and the weak solutions $g_j$ used in the construction are continuous, also the functions $f_j$ are continuous as a maximum of continuous functions such that (iii) is proved. Finally, (iv) follows since in the set $\{f_j>\psi\}$, the function $f_j$ is obtained as a maximum of a finite number of subsolutions $g_i$, $i\in \{ 0, ..., j-1\}$, and therefore it is a subsolution.
\end{proof}

\begin{lemma}\label{lemma32}
Take Construction \ref{construction} with an obstacle function $\psi\in C^0(\Omega_T)\cap L^\infty(\Omega_T)$. Then, the following statements for the limit function $u$ hold:
	\begin{enumerate}
		\item[(i)] The limit $u$ exists and satisfies $u\geq\psi$ in $\Omega_T$.
		\item[(ii)] If $v$ is an $m$-supercaloric function in $\Omega_T$ with $v\geq\psi$ in $\Omega_T$, then $v\geq u$ in $\Omega_T$.
		\item[(iii)] The function $u$ is lower semicontinuous and the set $\{ u>\psi \}$ is open.
	\end{enumerate}
\end{lemma}

\begin{proof}
Concerning (i), we observe that the limit $u$ exists at every point $(x,t)\in\Omega_T$ since the sequence $(f_j)_{j\in\N_0}$ is nondecreasing and uniformly bounded by Lemma \ref{properties_fj}. The assertion (ii) follows once we have shown that $v\geq f_j$ in $\Omega_T$ for any $j\in\N_0$. For $j=0$, there is nothing to prove since $f_0=\psi$ in $\Omega_T$. Assume now that there exists some $j\in\N_0$ such that $v\geq f_j$ holds. Then, by construction and the comparison principle, we obtain $v\geq g_j$ in $Q_j$, where $Q_j$ denotes the space-time box used in the construction. Therefore, we conclude that $v\geq f_{j+1}$ in $\Omega_T$, and induction proves that $v\geq u$ in $\Omega_T$. Moreover, as a limit of a nondecreasing sequence of continuous functions, $u$ is lower semicontinuous. This directly implies the openness of the set $\{ u>\psi \}$. \qedhere
\end{proof}

We proceed by gathering further properties of the limit function $u$. The next result contains a comparison principle.

\begin{lemma}\label{lemma33}
The limit $u$ from \eqref{limit_u} satisfies the comparison principle in all space-time boxes $Q\subset \Omega_T$, i.\,e.\ if $w\in C^0(\overline Q)$ is a weak solution to the porous medium equation in $Q$ satisfying $w\leq u$ on $\bd_p Q$, then we have $w\leq u$ in $Q$.
\end{lemma}

\begin{proof}
We fix a space-time box 
\[
Q=\prod_{i=1}^n (a_i,b_i) \times (t_1, t_2).
\]
Let $w\in C^0(\overline{Q})$ be a weak solution in $Q$ with $w\le u$ on $\bd_p Q$. Further, let $\eps>0$ and define the sets 
\[
E_j = \overline{Q} \cap \{f_j > w-\eps\}. 
\]
We observe that $E_j$ is open with respect to the relative topology by the continuity of $f_j$ and $w$. Moreover, we have $\bd_p Q \subset \bigcup_j E_j$ since $u\ge w $ on $\bd_p Q$. Therefore, we can choose $j_0\in\N$ such that $$f_{j_0}>w-\eps ~\textup{ on } \bd_p Q.$$ As the set $E_{j_0}$ is a neighbourhood of $\bd_p Q$ in the relative topology, there exists a number $j_1 \ge j_0$ such that $Q_{j_1}\in \F$ satisfies
\[
  \bd_p Q_{j_1} \cap \{t<t_2\} \subset E_{j_0} ~~\text{ and }~~ Q\setminus E_{j_0} \subset Q_{j_1}.
\]
The definition of $E_{j_0}$ and the monotonicity of the sequence $f_j$ guarantee that
\begin{equation}\label{eq:comparison-inequality}
w < f_{j_0} + \eps \le  f_{j_1} + \eps ~\textup{ on } \bd_p Q_{j_1} \cap \{t<t_2\}.
\end{equation}
In order to prove that $$w \le f_{j_1+1} ~\textup{ in } Q_{j_1} \cap \{t<t_2\},$$ we will show that $w \le g_{j_1}$. Since $g_{j_1}$ is a weak solution in $Q_{j_1}$ with boundary and initial values given by $f_{j_1}$, the inequality \eqref{eq:comparison-inequality} gives us $$w \le g_{j_1} + \eps ~\textup{ on } \bd_p Q_{j_1} \cap \{t<t_2\}.$$ Let $g_{j_1,\eps}$ be a weak solution in $Q_{j_1} \cap \{t<t_2\}$ with boundary and initial values $g_{j_1}+\eps$. The comparison principle implies $$w \le g_{j_1,\eps} ~\textup{ in } Q_{j_1} \cap \{t<t_2\}.$$ Therefore, we know
\[
(w-g_{j_1})_+(w^m-g_{j_1}^m)_+ \le (g_{j_1,\eps}-g_{j_1})(g_{j_1,\eps}^m-g_{j_1}^m).
\]
By Lemma \ref{cool-convergence}, we conclude 
\begin{align*}
  0 &\le \iint_{Q_{j_1} \cap \{t<t_2\}} (w-g_{j_1})_+(w^m-g_{j_1}^m)_+ \, dx \,dt  \\
&\le \iint_{Q_{j_1} \cap \{t<t_2\}}(g_{j_1,\eps}-g_{j_1})(g_{j_1,\eps}^m-g_{j_1}^m) \, dx \, dt \\
&\le C(\eps)
\end{align*}
with a constant $C(\eps)$ as in \eqref{cool-convergence-ineq}. Letting $\eps\to 0$, we have proven that $$w \le g_{j_1} ~\textup{ in } Q_{j_1} \cap \{t<t_2\}.$$ Thus, it follows $$w \le g_{j_1} \le f_{j_1+1}$$ and, by construction, also $w \le u$.  
\end{proof}

The comparison principle from Lemma \ref{lemma33} allows us to prove that the limit function $u$ is independent of the choice of the collection $\F$ used in Construction \ref{construction}. Thus, $u$ is well-defined. 

\begin{lemma}\label{box-lemma}
Let $\psi\in C^0(\Omega_T)\cap L^\infty(\Omega_T)$. Then, the limit $u$ is unique. In particular, it does not depend on the choice of space-time boxes used in Construction \ref{construction}.
\end{lemma}

\begin{proof}
Assume that $u^{(1)}$ and $u^{(2)}$ are two limits of the construction with generating functions $(f_j^{(i)})_{j\in\N_0}$ and $(g_j^{(i)})_{j\in\N_0}$ for $i\in\{1,2\}$. Moreover, we denote the corresponding collections of space-time boxes by $\{ Q_j^{(i)} \colon j\in\N_0 \}$. Obviously, we have $$u^{(1)} \geq f_0^{(1)} = \psi = f_0^{(2)}.$$ Now, suppose that $$u^{(1)} \geq f_j^{(2)} ~\textup{ in } \Omega_T$$ for some $j\in\N_0$. By construction, the function $g_j^{(2)}$ solves the boundary value problem \eqref{eq:BVP} in $Q_j^{(2)}$ with boundary and initial values $f_j^{(2)}$ on $\partial_p Q_j^{(2)}$. Since $$u^{(1)} \geq f_j^{(2)} = g_j^{(2)} ~\textup{ on } \partial_p Q_j^{(2)},$$ the comparison principle from Lemma \ref{lemma33} gives us $$u^{(1)} \geq g_j^{(2)} ~\textup{ in } Q_j^{(2)}.$$ Therefore, we have proven that $$u^{(1)} \geq \max \{ g_j^{(2)}, f_j^{(2)} \} = f_{j+1}^{(2)} ~\textup{ in } Q_j^{(2)},$$  and induction yields $u^{(1)} \geq u^{(2)}$ in $\Omega_T$. Eventually, interchanging the roles of $u^{(1)}$ and $u^{(2)}$, the claim follows.
\end{proof}

We conclude this section by establishing a comparison principle which displays that, if two obstacles $\psi^{(1)}$ and $\psi^{(2)}$ satisfy $\psi^{(1)}\leq\psi^{(2)}$, then, the associated limits are ordered in the same way.

\begin{lemma}\label{obstacle-comparison-lemma}
Suppose that $\psi^{(1)},\psi^{(2)} \in C^0(\Omega_T)\cap L^\infty (\Omega_T)$ are obstacles. If $\psi^{(1)}\leq \psi^{(2)}$ in $\Omega_T$, then the corresponding limits $u^{(1)}$ and $u^{(2)}$ of Construction \ref{construction} satisfy $u^{(1)}\leq u^{(2)}$ in $\Omega_T$.
\end{lemma}

\begin{proof}
By Lemma \ref{box-lemma}, the limits $u^{(1)}$ and $u^{(2)}$ do not depend on the choice of the collection $\F$. Thus, we may use the same family $\{ Q_j\colon j\in\N_0 \}$ of space-time boxes in the constructions of $u^{(1)}$ and $u^{(2)}$. Suppose that $u^{(i)}$ is generated by functions $f_j^{(i)}$ and $g_j^{(i)}$, where $i\in\{1,2\}$. By assumption, we have $$f_0^{(1)} = \psi_1\le \psi_2= f_0^{(2)} ~\textup{ in } \Omega_T.$$ Therefore, we may proceed by induction to show that the claim holds. Assume that $f_j^{(1)}\le f_j^{(2)}$ in $\Omega_T$ for some $j\in \N_0$. In particular, this implies $$f_j^{(1)}\le f_j^{(2)} ~\textup{ on } \bd_p Q_j.$$ Hence, by the comparison principle, we know that $g_j^{(1)}\le g_j^{(2)}$ in $ Q_j$, which in turn shows that $f_{j+1}^{(1)}\le f_{j+1}^{(2)}$ in $\Omega_T$. By induction, we conclude $u^{(1)} \le u^{(2)}$ in $\Omega_T$. 
\end{proof}

\section{Existence of solutions for lower semicontinuous obstacles} \label{sec:obstacle_lsc}

In this section, we prove our main result, Theorem \ref{main_theorem}. We remark that, unlike in Section \ref{sec:obstacle_cont}, we do not assume continuity for the obstacle $\psi$ anymore. 

\begin{theorem} \label{main_theorem}
Let $\psi$ be a non-negative, bounded, and lower semicontinuous obstacle function in $\Omega_T$. Then, there exists a unique solution $u$ to the obstacle problem in the sense of Definition \ref{def_obstacle_solution}.
\end{theorem}

\begin{proof}
Let $\psi$ be as in the statement of the theorem. There exists a nondecreasing sequence $(\psi_k)_{k\in\N}$ of continuous functions such that $$\psi_k\to \psi ~\textup{ a.\,e.\ in } \Omega_T.$$ Our aim is to construct a sequence $(u_k)_{k\in\N}$ of solutions to the obstacle problems with obstacles $\psi_k$ by using the results of Section \ref{sec:obstacle_cont}. We will show that the functions $u_k$ converge to a solution $u$ to the obstacle problem with obstacle $\psi$. In order to argue that the limit function is indeed a solution to the obstacle problem, we have to verify the properties for $u$ which are listed in Definition \ref{def_obstacle_solution}. 

First, Lemma \ref{lemma32}\,(i) yields $u_k\ge \psi_k$ in $\Omega_T$ for any $k\in\N$, and this property persists in the limit. Therefore, we have ensured (1) in Definition \ref{def_obstacle_solution}. Next, suppose that $v$ is an $m$-supercaloric function satisfying $$v\ge \psi\ge \psi_k ~\textup{ in } \Omega_T.$$ Then, by Lemma \ref{lemma32}\,(ii), we know that $u_k\le v$ in $\Omega_T$. Letting $k\to \infty$ shows that $u\le v$ in $\Omega_T$, i.\,e.\ (3) of Definition \ref{def_obstacle_solution}. 

It remains to verify (2). For that purpose, we first observe that the local boundedness of $u_k$ follows from Lemma \ref{properties_fj}\,(ii). Therefore, Lemma \ref{lemma32}\,(iii) and Lemma \ref{lemma33} show that $u_k$ is $m$-supercaloric for each $k\in\N$. We point out that by \cite[Lemma 4.5]{equivalence-def} (see also \cite[Lemma 4.1]{kkp}), the comparison principle for general subcylinders follows from the one proved for space-time boxes in Lemma \ref{lemma33}. By the comparison principle from Lemma \ref{obstacle-comparison-lemma}, the sequence $(u_k)_{k\in\N}$ is nondecreasing, and therefore, we may apply the Harnack type convergence theorem \cite[Prop.\ 6.8]{dichotomy} to conclude that $u_k$ converges to an $m$-supercaloric function $u$. Note that the finiteness condition (2) in Definition \ref{def:supercaloric} holds due to the boundedness of $\psi$. This proves that $u$ is a solution to the obstacle problem in the sense of Definition \ref{def_obstacle_solution}.
\end{proof}

\section{Applications}
In Theorem \ref{not-main_theorem}, we will show that Theorem \ref{main_theorem} gives us a method of approximating bounded $m$-supercaloric functions by a nonincreasing sequence of $m$-supercaloric functions which are bounded away from zero. In a sense, this complements the work of DiBenedetto et al.\ in \cite{dibenedetto-harnack}, where the existence of such a sequence was assumed. Moreover, we discuss the consequences of this result.



\begin{theorem}\label{not-main_theorem}
Let $u\ge 0$ be a bounded $m$-supercaloric function in $\Omega_T$. Then, there exists a sequence of uniformly bounded $m$-supercaloric functions $\{u_\eps\}_{\eps}$ satisfying $u_\eps \ge \eps$ in $\Omega_T$ for any $\eps>0$ such that, in the limit $\eps\to 0$, we have
\[
\begin{cases}
  u_\eps \to u ~~\text{pointwise in } \Omega_T, \\
  u_\eps^q \to u^q ~~\text{in } L^p(\Omega_T) ~ \text{ for }\,q>0,\,\, p\geq 1,\\
  u_\eps^q(\cdot,t_{0}) \to u^q(\cdot,t_{0}) ~~\text{in } L^p(\Omega) ~ \text{ for }\,t_0\in (0,T),\,\, q>0,\,\, p\geq 1,\\  
\nabla u_\eps^m \rightharpoondown \nabla u^m ~~\text{weakly in } L^2_{\rm loc}(\Omega_T,\R^n), \\
\nabla u_\eps^m \to \nabla u^m ~~\text{a.\,e.\ in }\Omega_T.
\end{cases}
\]
\end{theorem}

\begin{proof}
Without loss of generality, we may assume $\eps\in (0,1)$. For such a fixed number $\eps$, we define $$\psi_\eps = (u^m+\eps^m)^{1/m}.$$ Due to the fact that $u$ is a bounded, lower semicontinuous function in $\Omega_T$, we know that $\psi_\eps$ has those properties, too. Therefore, by Theorem \ref{main_theorem}, there exists a solution $u_\eps$ to the obstacle problem with obstacle $\psi_\eps$. Since the sequence $\psi_\eps$ is bounded and nonincreasing, also $u_\eps$ is bounded and nonincreasing. Thus, we conclude the pointwise convergence $u_\eps\to u$ as $\eps\to 0$. Using the dominated convergence theorem, this implies the convergence of $u_\eps^q$ in $L^p(\Omega_T)$. The same argument works on time-slices. Now, being a bounded $m$-supercaloric function, $u_\eps$ is also a weak supersolution (see Remark \ref{supercaloric-is-weak}), and therefore, the Caccioppoli inequality \cite[Lemma 2.15]{supersol} shows that $\nabla u_\eps^m$ is locally bounded in $L^2(\Omega_T)$ with an estimate independent of $\eps$. Hence, there exists a (nonrelabeled) subsequence of $\nabla u_\eps^m$ which converges weakly in $L^2_{\rm loc}(\Omega_T)$. 
By the arguments of \cite[Section 3.6]{sturm-DNPE}, the above mentioned convergences imply the pointwise convergence of the gradients $\nabla u_\eps^m$ a.\,e.\ in $\Omega_T$. We remark that weak solutions to doubly nonlinear equations were treated there so that the porous medium equation is covered as a special case. Moreover, the same reasoning applies also for weak supersolutions. 
\end{proof}

This result has several useful consequences. For instance, it gives us a method to improve reverse H\"older inequalities and weak Harnack estimates (see \cite[Thm.\ 1.1]{lehtela-harnack} and also \cite[Thm.\ 17.1]{dibenedetto-harnack}). In such estimates, assumptions of the form $u>0$ or $u\ge \varepsilon >0$ are often imposed. Usually, they are of merely technical nature as they arise from the application of a Moser type iteration method, where Caccioppoli estimates with possibly negative powers are utilized. Nevertheless, since constants cannot be added to (super)solutions to the porous medium equation, bypassing these assumptions is more involved in comparison to $p$-Laplace type equations. In \cite{dibenedetto-harnack}, an approximation by supersolutions which are bounded away from zero is suggested to overcome this issue. However, to the authors' knowledge, no method for such an approximation has been presented until now.

By applying Theorem \ref{not-main_theorem}, we can show that the weak Harnack estimate \cite[Thm.\ 1.1]{lehtela-harnack} holds for non-negative supersolutions to the porous medium equation. More precisely, an approximation by supersolutions $u_\eps\ge \eps>0$ allows us to discard the assumption $u>0$. As a consequence, the estimate takes the following form. 

\begin{corollary}\label{weak-harnack}
Let $u$ be a non-negative weak supersolution to \eqref{eq:PME} in $$\Omega_{T}\supset B(x_0,8\rho)\times(0,T).$$ Then, there exist constants $C_1,\,C_2>0$ depending on $m$ and $n$ such that, for almost every $t_0\in(0,T)$, the following inequality holds
\[
\fint_{B(x_0,\rho)} u(x,t_0)\dx x \le   \left( \frac{C_1 \rho^2}{T-t_0}\right ) ^{1/(m-1)} + C_2 \essinf_V u,
\]
where 
\begin{align*}
&V=B(x_0,4\rho) \times (t_0+\tau/2,\, t_0+\tau)\quad \text{and}\\
&\tau=\min \left\{ T-t_0,\, C_1\rho^2 \left ( \fint_{B(x_0,\rho)} u(x,t_0) \dx x \right )^{-(m-1)}\right\}.
\end{align*}
\end{corollary}



\end{document}